\documentclass[12pt]{amsart}
\usepackage{graphicx,latexsym,amsfonts,amsmath,amssymb,amsthm,url,amsbsy,amscd}
\usepackage[english]{babel}
\usepackage[latin1]{inputenc}
\usepackage{graphicx,psfrag,epsfig}
\usepackage{enumerate}
\usepackage{dsfont}
\usepackage[normalem]{ulem}
\newtheorem{theorem}{Theorem}
\newtheorem{lemma}[theorem]{Lemma}
\newtheorem{proposition}[theorem]{Proposition}

\newtheorem{definition}[theorem]{Definition}
\newtheorem{remark}[theorem]{Remark}

\newcommand{\dd}{\;{\rm d}}
\newcommand{\ddo}{\;{\rm d}}
\newcommand{\xx}{{\bf x}}
\newcommand{\zz}{{\bf Z}}
\newcommand{\vv}{{\bf v}}
\newcommand{\EE}{{\mathbb E}}
\newcommand{\C}{{\mathcal C}}
\renewcommand{\SS}{{\mathbb S}}
\begin{document}
\title{Topological interactions in a Boltzmann-type framework}

\author[A. Blanchet, P. Degond]{Adrien Blanchet$^1$ \& Pierre Degond$^2$}
\address{$^1$ GREMAQ (IAST/TSE, UT1) -- 21 All\'ee de Brienne 31000 Toulouse, France\\
\email{Adrien.Blanchet@ut-capitole.fr}\\
$^2$ Department of Mathematics, Imperial College London, London SW7 2AZ, United Kingdom\\
\email{p.degond@imperial.ac.uk}\\
}

\date{\today}



\maketitle
\begin{abstract}
We consider a finite number of particles characterised by their positions and velocities. At random times a randomly chosen particle, the follower, adopts the velocity of another particle, the leader. The follower chooses its leader according to the proximity rank of the latter with respect to the former. We study the limit of a system size going to infinity and, under the assumption of propagation of chaos, show that the limit equation is akin to the Boltzmann equation. However, it exhibits a spatial non-locality instead of the classical non-locality in velocity space. This result relies on the approximation properties of Bernstein polynomials.
\end{abstract}
\section{Introduction}
In this paper, we explore collective dynamics driven by rank-based
interactions, {\it i.e.} that's to say interactions determined by the
rank of the agents with respect to certain criterion. There are many
examples where such interactions take place. In economics for instance,
it was extensively analysed in~\cite{harsanyi1976cardinal} that agents
are more sensitive to their rank compared to others (salary or wealth
for example) than their own independant cardinal level. To go
further,~\cite{lazear1981rank} studies, in an organisation, compensation
schemes which pay according to an individual's ordinal rank rather than
their output level. Such payoff based on the rank approach also appears
very naturally in a variety of economics applications such as bids, the
labour market, portfolio management, the oil market, academic
production, reputation, etc. 

Evolutionary game theory studies the evolution of strategies/genes
transmitted through natural selection. In chimpanzees as in cockroaches
a group is formed of a dominant male, females and lower order males.
Only the dominant male is supposed to mate with the females. However,
when the dominant male is absent, the females also reproduce with other
males giving the preference to males in descending
order~\cite{dewsbury1982dominance}. It is also known that the rank of an
offspring strongly depends on the rank of its
mother~\cite{horrocks1983maternal}, so that in the replicator dynamic
process the rank increases the chance of reproduction. The study of such
models requires taking into account interactions depending on the rank
of the agents.

In this article we focus on the dynamics of bird flocks. There is a
wide\-spread literature of flocking models where the birds react to
their neighbours as a function of the neighbours' distance from them
within the flock. These are the co-called ``metric'' interactions. In
this context, dynamics based on alignment~\cite{vicsek1995novel},
consensus~\cite{cucker2007emergent} or attraction-repulsion
see~\cite{bertozzi2009blow,bertozzi2007finite} have been widely studied.
However, there has been recent compelling
evidence~\cite{ballerini2008interaction} that interactions within bird
flocks are mostly metric free, as the birds react primarily with a
limited number of their nearest neighbours irrespective of the distances
between them. This observation has motivated the concept of
``topological interaction'', which has been widely echoed in the
scientific
literature~\cite{bode2010limited,camperi2012spatially,cavagna2008starflag,ginelli2010relevance,niizato2011metric}.

Our goal is to investigate the large size limit of a system of agents
interacting through topological interactions. Specifically, we consider
a leader-follower model~\cite{carlen2013kinetic2,carlen2013kinetic1}
where at random times a randomly chosen bird, the follower, decides to
adopt the velocity of another bird, its leader, in the flock. The
follower chooses its leader according to a probability only depending on
the proximity rank of the latter with respect to the former. If we
assume that the probability has a strong cutoff as soon as the proximity
rank exceeds a certain value, of the order of seven in actual flocks,
the considered model is akin to the topological interactions
of~\cite{ballerini2008interaction}.

To our knowledge,~\cite{brenier20092} is the first mathematical work where interaction rules between agents depending on their rank are considered. The closest to our work is~\cite{haskovec2013flocking} where kinetic and hydrodynamic models for topological interactions have been proposed. However, the considered dynamics is different from ours. In \cite{haskovec2013flocking}, it is supposed that an agent's velocity relaxes towards an average velocity of its neighbours where the relative weights of the neighbours depend on their proximity rank to the considered agent. Therefore, it is a model of Cucker-Smale type~\cite{cucker2007emergent} combined with a topological interaction rule. In \cite{haskovec2013flocking}, a mean-field type kinetic model is rigorously derived under some regularisation in the large system size limit and a hydrodynamic model under a monokinetic closure assumption is proposed. 

Here, the interaction rule is different and, in the large system size limit, leads to a Boltzmann type model with an integral operator describing the balance between gains and losses due to the interactions rather than a mean-field model where the interactions are described through a force field. From the mathematical viewpoint, this makes a considerable difference, as an empirical measure approach is not possible. Instead, one has to rely on the propagation of chaos property for the solution of the master equation. In the present work, propagation of chaos is assumed and its proof is defered to future work. Still, under this assumption, the derivation of the kinetic equation is not obvious and as we will see, relies on fine approximation properties of Bernstein polynomials.

Indeed, we will realise that the derivation of a kinetic model requires the estimation of the probability that given two particles say numbered $1$ and~$2$, the rank of $2$ with respect to $1$ be equal to a given integer $j$. Then, the interaction probability of $1$ with $2$ in this configuration is a function $K(j/N)$, where $N$ is the total number of particles and the function $K$ is characteristic of the considered interaction. Thanks to an easy combinatorial estimation, the total probability of $1$ interacting with $2$ is found as  the Bernstein polynomial approximation of $K$ when $N$ is large. Due to some cancellations, the first order correction in powers of $1/N$ of the Bernstein polynomial approximation of $K$ is also needed. This correction can be found in the literature~\cite{lorentz2012bernstein}. 

The paper is organised as follows. In Section~\ref{ndyn}, we present the $N$ particle dynamics and state the main result. In Section~\ref{mastereq}, we derive the master equation of the process and the equation for the first marginal under the assumption of propagation of chaos. In Section~\ref{sec:limit} we precisely state our main result, namely that, in the limit $N \to \infty$ and the assumption of propagation of chaos, the equation for the first marginal reduces to a kinetic equation of Boltzmann type with spatial nonlinearity. To prove this theorem, we use results on Bernstein's polynomial approximation from the literature~\cite{lorentz2012bernstein}. Section \ref{sec:conclusion} offers some considerations on the limit kinetic equation and illustrates our discussion with numerical simulations. Finally, a conclusion is drawn in Section \ref{sec:conclusion2}.

\section{The $N$-particle dynamics}\label{ndyn}
Consider a set of $N$ particles. The particle $i$ is characterised by its position $x_i \in {\mathbb R}^n$ and its velocity $v_i \in {\mathbb R}^n$ where $n\geq 1$ is both the spatial and velocity dimension. For a given particle $i$ we can order the other particles relatively to their distance to $i$. More precisely, we have the following:
\begin{definition}[Rank]
Consider $N$ particles located at $x_1, \ldots, x_N$. Consider the $i$-th particle and order the list $\big(|x_j-x_i|\big)_{j=1, \ldots, N, \, j \not = i}$ by increasing order and denote by $R^N(i,j)\in \{1, \ldots,  N-1\}$ the position of the $j$-th item in this list. If two indices $j$ and $j'$ are such that  $|x_j-x_i| = |x_{j'}-x_i|$, then we choose arbitrarily an ordering between these two numbers. We define $R^N(i,i)=0$. Now, we define \emph{the rank of $j$ with respect to $i$} as: 
$$ r^N(i,j) = \frac{R^N(i,j)}{N-1} \quad \in \quad \bigcup_{k=1}^{N-1} \Big\{ \frac{k}{N-1} \Big\}. $$
\end{definition}
We introduce a function $K$: $r \in [0,1] \mapsto K(r) \in [0,\infty)$ such that 
$$ \int_0^1 K(r) \dd r = 1\;. $$
We define
$$ K^N (r) = \dfrac{K(r)}{\sum_{k=1}^{N-1} K \Big( \frac{k}{N-1} \Big)}\,, $$
in order to have for any $i \in \{1, \ldots, N\}$: 
$$ \sum_{\substack{j=1\\j\not=i}}^N K^N\big(r^N(i,j) \big) = \sum_{k=1}^{N-1} K^N \left( \frac{k}{N-1} \right) = 1\;. $$
In this way, for any $i \in \{1, \ldots, N\}$, the collection $(\pi_{ij})_{j=1, j \not = i}^N$, where 
$$ \pi_{ij}^N = K^N \big( r^N(i,j) \big), $$
defines a discrete probability measure on the set $\{j\in \{1, \ldots, N\}, \, \, j \not = i\}$. 
\medskip

Consider $N$ particles $\{ (x_1(t), v_1(t)), \ldots, (x_N(t), v_N(t)) \}$ which are subject to the following dynamics (previously referred to as the ``Choose the Leader'' dynamics \cite{carlen2013kinetic2,carlen2013kinetic1}): 
\begin{itemize}
\item[-] The dynamics is a succession of free-flights and collisions. 
\item[-] During free-flight, particles follow straight trajectories 
\begin{equation*}
\left\{
  \begin{array}{l}
    \dot x_i = v_i, \vspace{.1cm}\\
 \dot v_i = 0\;.
  \end{array}
\right.
\end{equation*}
\item[-] At Poisson random times with a rate equals to $N$, particles undergo the following collisions process: Pick a particle $i$ in $\{1, \ldots, N\}$ with uniform probability ${1}/{N}$ and perform a collision, \emph{i.e.} pick a collision partner $j$ in the set $\{j\in \{1, \ldots, N\}, \, \, j \not = i\}$ with probability $\pi_{ij}^N$ and perform: 
\begin{equation*}
\left\{
  \begin{array}{l}
	(x_i,x_j) \mbox{ remains unchanged, } \vspace{.1cm}\\
 (v_i,v_j) \mbox{ is changed into } (v_j,v_j). 
  \end{array}
\right.
\end{equation*}
\end{itemize}\medskip 

Since the rank of $j$ with respect to $i$ is an intrinsic property of the positions of the pair of particles and does not depend on how they are numbered, we have the following properties of the rank:
\begin{remark} Let $(x_1, \ldots, x_N)$ be a set on $N$ particles.
  \begin{itemize}
  \item[(i)] The rank $r^N(i,j)$, and hence $\pi_{ij}^N$, is a function of $(x_1, \ldots, x_N)$, \emph{i.e.}  $$r^N(i,j) = r^N(i,j)(x_1, \ldots, x_N)\;.$$ More precisely, we consider the rank $r^N(i,j)$ as a function of $L^\infty({\mathbb R^{nN}})$.
  \item[(ii)] The rank is permutation invariant, {\it i.e.} for any permutation $\sigma \in \mathfrak{S}_N$ where $\mathfrak{S}_N$ denotes the set of permutations of $\{1,\ldots,N\}$, we have
\begin{equation*}
 r^N(\sigma(i), \sigma(j))(x_{\sigma(1)}, \ldots, x_{\sigma(N)}) = r^N(i,j)(x_1, \ldots, x_N). 
\end{equation*}
  \end{itemize}
\label{lem:rank} 
\end{remark}

The aim of this article is to study the limit of this dynamics when the number of particles goes to $\infty$. To do so we will assume that the propagation of chaos property holds true {\it i.e.}
\begin{equation*}
  f^{(N)}(Z_1,\cdots,Z_N, t)=\prod_{\ell=1}^N f^{(1)}_N(Z_\ell,t), \quad \forall \zz \in {\mathbb R}^{2nN}, \quad \forall t \in [0,\infty)\;.
\end{equation*}
Assuming that $f^{(1)}_N \to f$ and $\rho^{(1)}_N:=\int f^{(1)}_N\dd v\to \rho = \int  f \dd v$, then in the limit $N \to \infty$, we will prove that $f$ is a solution of the kinetic equation:
\begin{equation*}
  \frac{\partial f}{\partial t}(x,v)+v\cdot \nabla_x f(x,v)=\rho(x) \int f(x',v)\,K\left( M_\rho (x,|x'-x|)\right)\dd x' - f(x,v), 
\end{equation*}
where $M_\rho$ is the partial mass of $\rho$ and is defined by
\begin{equation*}
  M_\rho(x,s)=\int_{x' \in B(x,s)} \rho(x')\dd x'\;,
\end{equation*}
and where $B(x,s) = \{ y \in {\mathbb R}^n \, | \, |y-x| \leq s \}$ is the ball centered at $x$ and of radius $s>0$. 
\begin{remark}
 The conservation of mass property holds true by Lemma~\ref{lem:change} applied to $H=K$.
\end{remark}
\medskip

In the following section, we derive the master equation for this process, Section \ref{subsec:master}, and the first marginal equation for indistinguishable particles, Section \ref{subsec:firstmarginal}. Then, in Section~\ref{sec:propa}, we derive the master equation under the assumption of propagation of chaos. 
\section{Master equation and propagation of chaos}\label{mastereq}

\subsection{Master equation} \label{subsec:master}
To simplify the notation, when no confusion is possible, we will denote $\xx:=(x_1,\ldots,x_N)$,  $\vv:=(v_1,\ldots,v_N)$, $Z_i:=(x_i,v_i)$, $\zz:=(Z_1,\ldots,Z_N)$ and $\ddo \zz:=\dd x_1 \dd v_1\ldots\dd x_N \dd v_N$.\smallskip

As the collisions occur at Poisson times with rate $N$, the master equation in weak form is, for all test function $\phi^N:\zz \mapsto \phi^N(\zz)$:
\begin{align}
  \label{eq:master}
  &\partial_t \int f^{(N)}(\zz)\,\phi^N(\zz)\ddo \zz - \sum_{i=1}^N \int f^{(N)}(\zz) (v_i \cdot \nabla_{x_i})\,\phi^N(\zz)\dd \zz \notag\\
&=\! N\int \Bigg[ \frac1N\! \sum_{\substack{i,j=1\\j\not=i}}^N  \pi_{ij}^N(\xx)  \,\phi^N(Z_1, \ldots, x_i, v_j , \ldots, x_j, v_j , \ldots Z_N) \notag\\
&\hspace{7cm} -\phi^N(\zz)\Bigg]f^{(N)}(\zz) \ddo \zz\notag\\	
&= N \int\Bigg[ \frac1N\sum_{\substack{i,j=1\\j\not=i}}^N \int \pi_{ij}^N(\xx)\,\phi^N(Z_1, \ldots, x_i, v'_i , \ldots, x_j,v_j, \ldots Z_N)\,\delta(v'_i-v_j)\dd v'_i\notag\\
&\hspace{7cm}-\phi^N(\zz)\Bigg]  f^{(N)}(\zz)\ddo \zz.
\end{align}
By exchanging the notations $v_i$ and $v'_i$ we obtain the following master equation in the strong form:
\begin{equation*}
\partial_t f^{(N)}(\zz) = \sum_{i=1}^N f^{(N)}(\zz) \,(v_i \cdot \nabla_{x_i}) + N L f^{(N)}(\zz) \;,
\end{equation*}
where the operator $L$ is defined by
\begin{multline*}
L f^{(N)}(\zz) :=\\ \frac1N  \sum_{\substack{i,j=1\\i\not=j}}^N  \pi_{ij}^N(\xx)\, \delta(v_i-v_j) \! \int  f^{(N)}(Z_1, \ldots, x_i, v'_i , \ldots Z_N) \dd v'_i-f^{(N)}(\zz).
\end{multline*}

\begin{lemma}[Invariance under permutation]
Define for all $\sigma \in \mathfrak{S}_N$, 
$$\sigma f^{(N)} (\zz) := f^{(N)} (\zz_{\sigma(1)}, \ldots, \zz_{\sigma(N)} ).$$ 
Then we have:
$$ L (\sigma f^{(N)}) = \sigma \big( L f^{(N)} \big) . $$
As a consequence, if $f^{(N)}(t)$ is permutation invariant at time $t=0$, \emph{i.e.} $\sigma f^{(N)}(t)|_{t=0} = f^{(N)}(t)|_{t=0}$ for all $\sigma \in \mathfrak{S}_N$, then it is permutation invariant for all times.
\end{lemma}
\begin{proof} To emphasise the dependence in $\zz$, we can rewrite the operator $L$ as: 
\begin{eqnarray*}
L f^{(N)} (\zz)=  \frac1N \sum_{\substack{i,j=1\\i\not=j}}^N \pi_{ij}^N(\vec\xx(\zz)) \delta(V_i(\zz) -V_j(\zz))  P_i f^{(N)}(\zz) -f^{(N)}(\zz), 
\end{eqnarray*}
with $\vec\xx(\zz)= \xx$, $ V_i(\zz) = v_i$ and 
$$ P_i f^{(N)}(\zz) = \int  \,f^{(N)}(Z_1, \ldots, x_i, v'_i , \ldots Z_N) \dd v'_i\;. $$
First note that, setting $\sigma \zz=(Z_{\sigma(1)}, \ldots, Z_{\sigma(N)} )$, we have 
$$ V_i (\zz) = V_{\sigma(i)}(\sigma \zz)\,, \quad\mbox{and}\quad  P_i (\sigma f^{(N)}) (\zz) = P_{\sigma(i)} f^{(N)} (\sigma \zz)\;. $$
Therefore by applying the $\sigma^{-1}$ permutation to the double sum and using the permutation invariance of the rank, see Lemma~\ref{lem:rank} (ii), we obtain
\begin{align*}
L \sigma f^{(N)} (\zz) &=  \frac1N \sum_{\substack{i,j=1\\i\not=j}}^N \pi_{ij}^N(\vec \xx(\zz)) \,  \delta(V_i (\zz) -V_j (\zz))  \, P_i (\sigma f^{(N)}) (\zz) \\
&\hspace{3.5cm}-(\sigma f^{(N)})(\zz)\\ 
&=  \frac1N \sum_{\substack{i,j=1\\i\not=j}}^N K^N[r^N(i,j)](\vec \xx(\zz)) \,  \delta(V_i (\zz) -V_j (\zz))  \, P_{\sigma(i)} f^{(N)} (\sigma \zz)\\
&\hspace{3.5cm} -(\sigma f^{(N)})(\zz)\\ 
&=\frac1N \sum_{\substack{i',j'=1\\i'\not=j'}}^N K^N[r^N(\sigma^{-1}(i'),\sigma^{-1}(j'))](\vec \xx(\zz)) \, \delta(V_{\sigma^{-1}(i')} (\zz)\\
&\hspace{3.5cm} -V_{\sigma^{-1}(j')} (\zz)) P_{i'}  f^{(N)} (\sigma\zz) -(\sigma f^{(N)})(\zz)\\ 
& = \frac1N \sum_{\substack{i',j'=1\\i'\not=j'}}^N K^N[r^N(i',j')] (\vec \xx(\sigma\zz) ) \, \delta(V_{i'} (\sigma \zz)\\
&\hspace{3.5cm} -V_{j'} (\sigma\zz)) P_{i'}  f^{(N)} (\sigma\zz) -f^{(N)}(\sigma \zz)\\ 
& = (L f^{(N)}) (\sigma\zz ) = \sigma L \partial_t f^{(N)}  (\zz). 
\end{align*}
The above property states that $\sigma^{-1} L \sigma = L$, for all $\sigma \in \mathfrak{S}_N$. Supposing that $L$ is a bounded operator, we deduce that $\sigma^{-1} L^k \sigma = L^k$, for all $k \in {\mathbb N}$ and consequently $\sigma^{-1} e^L \sigma = e^L$. Now, the solution of the problem $\partial_t f^{(N)} = NLf^{(N)}$ with $f^{(N)}|_{t=0} = f^{(N)}_0$ can be written $f^{(N)}(t) = e^{NLt} f^{(N)}_0$. We deduce that $\sigma f^{(N)}(t) = e^{NLt} \sigma f^{(N)}_0$. Therefore, if $\sigma f^{(N)}_0= f^{(N)}_0$, then $\sigma f^{(N)}(t)= f^{(N)}(t)$, for all $t \geq 0$. If $L$ is not bounded, the same property remains true thanks to an approximation argument. 
\end{proof}
\subsection{First marginal equation for indistinguishable particles}
\label{subsec:firstmarginal}
In the remainder of this article, we will suppose that $f^{(N)}$ is invariant under permutations which physically means that the particles are indistinguishable. This allows us to define the $k$-particle marginal as
\begin{equation}
  \label{eq:assumptionf2}\tag{A1}
  f^{(k)}_N(Z_1, \ldots, Z_k,t)=\int f^{(N)}(Z_1,\cdots,Z_N,t)\dd Z_{k+1} \cdots \dd Z_N .
	\end{equation}
	and $f^{(k)}_N$ is still invariant under permutations of $(Z_1, \ldots, Z_k)$.
\begin{proposition}[First marginal equation for indistinguishable particles]
Assume \eqref{eq:assumptionf2}. For any test functions satisfying 
\begin{equation}
  \label{eq:assumptionphi}
  \phi^N(Z_1,\cdots,Z_N)=\phi(Z_1)\,,
\end{equation}
we have
      \begin{align*}
   \partial_t \int f^{(1)}_N(Z_1)\,\phi(Z_1) \dd Z_1 &=\sum_{i=1}^N \int f^{(1)}_N(\zz) \,(v_i \cdot \nabla_{x_i})\,\phi(\zz)\dd \zz\notag \\
&\quad+ (N-1) \int \pi_{12}^N(\xx) \,\phi(x_1,v_2) f^{(N)} (\zz) \dd \zz \notag \\
	&\quad+ (N-1) \int \pi_{21}^N(\xx) \,\phi(Z_1) f^{(N)} (\zz) \dd \zz \notag \\
&\quad+ (N-1)(N-2) \int \pi_{23}^N(\xx) \,\phi(Z_1) f^{(N)} (\zz) \dd \zz \notag \\
&\quad- N \int \phi(Z_1) \, f^{(1)}_N(Z_1) \dd Z_1. 
\end{align*}    
\label{prop:marginal1}
\end{proposition}
\begin{proof}
Separating the cases $i=1\neq j$, $j=1\neq i$, and $i\ge2$, $j\ge 2$, the master equation \eqref{eq:master} gives
\begin{multline}
  \label{eq:abc}
   \partial_t \int f^{(N)}(\zz)\,\phi(\zz) \dd \zz =  \sum_{j=2}^N A_{j}^{(1)}+\sum_{i=2}^N A_{i}^{(2)} + \sum_{i=2}^N \sum_{j=2, j \not = i}^N A_{i,j}\\- N\int\phi(x_1,v_1) f^{(N)}(\zz) \dd \zz ,
\end{multline}
with
\begin{align*}
  A_{j}^{(1)}&:=\int \pi_{1j}^N (\xx) \,\phi(x_1,v_j)f^{(N)}(\zz) \dd \zz,\\
 A_{i}^{(2)}&:=\int \pi_{i1}^N (\xx) \,\phi(Z_1)f^{(N)}(\zz) \dd \zz,\\
A_{i,j}&:=\int \pi_{ij}^N(\xx)\, \phi(Z_1)f^{(N)}(\zz) \dd \zz .
\end{align*}
To compute the first term $A_{j}^{(1)}$, we perform the change of variables $Z'_2 = Z_j$ and $Z'_j = Z_2$, which leads to:
\begin{multline*}
A_{j}^{(1)}\!
 = \! \int \pi_{1j}^N(x_1,x'_j, \ldots,x'_2, \ldots ,x_N) \,\phi(x_1,v'_2)\, f^{(N)} (Z_1,Z'_j,\ldots,Z'_2,\ldots, Z_N) \\
\dd Z_1 \dd Z'_j  \ldots \dd Z'_2 \ldots \dd Z_N \,.
\end{multline*}
Using the permutation invariance of the rank (see Lemma~\ref{lem:rank} (ii)), we have
$$\pi_{1j}^N(x_1,x'_j, \ldots,x'_2, \ldots ,x_N) = \pi_{12}^N(x_1,x'_2, \ldots,x'_j, \ldots ,x_N) . $$
Therefore, dropping the primes and using the permutation invariance of $f^{(N)}$, we obtain 
\begin{equation*}
A_{j}^{(1)}=  \int \pi_{12}^N(\xx) \,\phi(x_1,v_2) f^{(N)} (\zz) \dd \zz,  \notag 
\end{equation*}
which does not depend on $j$.

Similarly, we have 
\begin{multline*}
A_{i}^{(2)}
 =  \int \pi_{i1}^N(x_1,x_i, \ldots,x_2, \ldots ,x_N) \,\phi(Z_1)\, f^{(N)} (Z_1,Z_i,\ldots,Z_2,\ldots, Z_N) \\
\dd Z_1 \dd Z_i  \ldots \dd Z_2 \ldots \dd Z_N \, 
\end{multline*}
so that,  using the permutation invariance of the rank, see Lemma~\ref{lem:rank} (ii), and the permutation invariance of $f^{(N)}$ as previously, we obtain 
\begin{equation*}
A_{i}^{(2)}=  \int \pi_{21}^N(\xx) \,\phi(Z_1) f^{(N)} (\zz) \dd \zz \,,
\end{equation*}
which does not depend on $i$.

Also, we have with $i \geq 2$, $j \geq 2$ and $i \not = j$:
\begin{multline*}
A_{i,j}=
 \int \pi_{ij}^N(x_1,x_i, x_j\ldots,x_2, \ldots , x_3, \ldots ,x_N) \,\phi(Z_1)\\
 f^{(N)} (Z_1,Z_i, Z_j,\ldots,Z_2,\ldots, Z_3, \ldots Z_N) \dd Z_1 \dd Z_i  \dd Z_j \ldots \!\dd Z_2 \ldots \!\dd Z_3 \ldots \!\dd Z_N . 
\end{multline*}
Then using the permutation invariance of the rank, see Lemma~\ref{lem:rank} (ii), and the permutation invariance of $f^{(N)}$ as previously, we obtain 
\begin{equation*}
A_{i,j}=  \int \pi_{23}^N(\xx) \,\phi(Z_1) \,f^{(N)} (\zz) \dd \zz \,.
\end{equation*}

For the last term of~\eqref{eq:abc} we obviously have  
$$ \int\phi(Z_1) f^{(N)}(\zz) \dd \zz= \int \phi(Z_1) \, f^{(1)}_N(Z_1) \dd Z_1. $$

Collecting all these identities, we obtain the identity stated in Proposition~\ref{prop:marginal1}. 
\end{proof}
\subsection{Propagation of chaos}\label{sec:propa}
Assume now that the propagation of chaos property holds true {\it i.e.}
\begin{equation}
  \label{eq:assumptionf}\tag{A2}
  f^{(N)}(Z_1,\cdots,Z_N, t)=\prod_{\ell=1}^N f^{(1)}_N(Z_\ell,t), \quad \forall \zz \in {\mathbb R}^{2nN}, \quad \forall t \in [0,\infty)\;,
\end{equation}
and define:
$$\rho^{(1)}_N (x) = \int f^{(1)}_N (x,v) \dd v\;.$$
We remark that $\rho^{(1)}_N$ is a probability density.

We have the following proposition:  
\begin{proposition}[First marginal equation with propagation of chaos]
Assume \eqref{eq:assumptionf}. For any test functions satisfying \eqref{eq:assumptionphi}, we have
\begin{multline}
  \label{eq:master2_poc}
  \partial_t \int f^{(1)}_N(Z_1)\,\phi(Z_1) \dd Z_1 \\
= \sum_{i=1}^N \int f^{(1)}_N(\zz)\, (v_i \cdot \nabla_{x_i})\,\phi(\zz)\dd \zz +  (A^N) + (B^N) + (C^N) + (D^N) ,
\end{multline}    
with
\begin{align*} 
(A^N)&=\frac{1}{S^N(K)} \int \phi(x_1,v_2) \, f^{(1)}_N (Z_1) \, f^{(1)}_N (Z_2) \\
&\hspace{2.5cm} K\left(r^N(1,2)(\xx)\right) \prod_{\ell=3}^N \rho^{(1)}_N(x_\ell) \dd x_\ell \dd Z_1 \dd Z_2, \notag \\
(B^N)&= \frac{1}{S^N(K)}\int \phi(Z_1) \, f^{(1)}_N (Z_1) K\left(r^N(2,1)(\xx)\right)  \prod_{\ell=2}^N \rho^{(1)}_N(x_\ell) \dd x_\ell  \dd Z_1,  \notag  \\
(C^N)&= \frac{N-2}{S^N(K)} \int \phi(Z_1) \, f^{(1)}_N (Z_1)  K\left(r^N(2,3)(\xx)\right)  \prod_{\ell=2}^N \rho^{(1)}_N(x_\ell) \dd x_\ell \dd Z_1,  \notag  \\
(D^N)&=- N \int \phi(Z_1) \, f^{(1)}_N(Z_1) \dd Z_1, 
\end{align*}
where $S^N(K) $ is given by
$$ S^N(K) = \frac{1}{N-1} \sum_{k=1}^{N-1} K\left(\frac{k}{N-1} \right) \;. $$  
\end{proposition}
We note that $S^N(K) $ is the Riemann sum approximation of $\int_0^1 K(r) \dd r$. Since we assume $\int_0^1 K(r) \dd r = 1$, $S^N(K)$ converges to $1$ as $N$ goes to $\infty$.
\begin{proof}
  This result is a direct consequence of Proposition \ref{prop:marginal1}, integrating in $v$ when possible. We then use that
  \begin{equation*}
    K^N=\frac{K}{(N-1)S_N(K)}\;,
  \end{equation*}
to obtain the stated result.
\end{proof}
\section{Limit equation}\label{sec:limit}
For a density $\rho$, define the partial mass of $\rho$ centred in $x$ and of radius $s$ by:
\begin{equation*}
M_\rho (x, s) = \int_{|x-x'| \leq s} \rho(x') \dd x'\;. 
\end{equation*} 
We now state the main theorem of this article:
\begin{theorem}[Limit equation]
Assume \eqref{eq:assumptionf}. If 
$$\lim_{N \to \infty}f^{(1)}_N \to f \quad\mbox{and}\quad \lim_{N \to \infty}\rho^{(1)}_N \to \rho = \int  f \dd v,$$
then, in the limit $N \to \infty$, for all test functions $\phi$ we have:
\begin{multline*}
  \partial_t \int f(\zz) \,\phi(\zz)\dd \zz \\= \int \phi(x_1,v_2) \, f (Z_2)\,\rho(x_1)\,K\left( M_{\rho}(x_1,|x_2-x_1|\right)\dd x_1\dd Z_2 \\-\int \phi(Z_1) \, f (Z_1)\dd Z_1\;,
\end{multline*}
or, in strong form:
\begin{equation*}
  \frac{\partial f}{\partial t}(x,v)+v\cdot \nabla_x f(x,v)=\rho(x) \int f(x',v)\,K\left( M_\rho (x,|x'-x|)\right)\dd x' - f(x,v). 
\end{equation*}
\label{thm:limit}
\end{theorem}

This result will be obtained by passing to the limit when $N \to \infty$ in~\eqref{eq:master2_poc}. To pass to the limit in the transport term of~\eqref{eq:master2_poc} is classical and we refer the reader to classical textbooks on the subject. We divide the proof of this theorem in two sections. The first section deals with the two first terms $(A^N)$ and $(B^N)$ of~\eqref{eq:master2_poc} while the second will deal with the last two terms $(C^N)$ and $(D^N)$ of~\eqref{eq:master2_poc}. To do so we will be using the Bernstein polynomial approximation of functions which is a follows:
\begin{proposition}[Bernstein polynomial approximation, \cite{lorentz2012bernstein}]
\label{bernstein}  Let $f$ be a function defined on $[0,1]$. The $n$-th Bernstein polynomial associated with $f$ is defined by
  \begin{equation*}
    B_n(f;x):=\sum_{i=0}^n f\left(\frac{i}{n} \right) {n \choose i} x^i(1-x)^{n-i}\;.
  \end{equation*}
   If $f \in \C^2[0,1]$ then 
\begin{equation*}
  B_n(f;x)=f(x)+\frac{x(1-x)}{2n}f''(x)+o\left(\frac{1}{n}\right)\;.
\end{equation*}
\end{proposition}
\subsection{Evaluation of $(A^N)$ and $(B^N)$}
\begin{proposition}[Evaluation of $(A^N)$ and $(B^N)$]\label{cor:anbn} Under the assumptions of Theorem~\ref{thm:limit}, we have for large $N$
  \begin{multline*}
S^N(K)\times (A^N)= \\
\int \phi(x_1,v_2) \, \rho^{(1)}_N (Z_1) \, f^{(1)}_N (Z_2) \,  K\left(M_{\rho^{(1)}_N} (x_1, |x_1-x_2|)\right)  \dd x_1 \dd Z_2 +o(1),
  \end{multline*}
and
  \begin{multline*}
S^N(K)\times(B^N)= \\
\int \phi(Z_1) \, f^{(1)}_N (Z_1)  \,  K\left(M_{\rho^{(1)}_N} (x_2, |x_1-x_2|)\right)\rho^{(1)}_N (x_2) \dd Z_1 \dd x_2 +o(1).
  \end{multline*}
\end{proposition}

To prove this result we first prove the following:
\begin{lemma}\label{lem:1221} Under the assumptions of Theorem~\ref{thm:limit}, we have for $N$ large,
\begin{equation*}
\int K\left(r^N(1,2)(\xx)\right)  \prod_{\ell = 3}^N \rho^{(1)}_N (x_\ell) \dd x_\ell = K\left(M_{\rho^{(1)}_N} (x_1, |x_1-x_2|)\right) +o(1)\,,
\end{equation*}
and 
\begin{multline*}
\int K\left(r^N(2,1)(\xx)\right) \prod_{\ell = 2}^N \rho^{(1)}_N (x_\ell) \dd x_\ell \\=\int K\left(M_{\rho^{(1)}_N} (x_2, |x_1-x_2|)\right)\rho^{(1)}_N (x_2) \dd x_2+o(1)\;.
\end{multline*}
\end{lemma}
\begin{proof} We first give a combinatorial interpretation of the rank and then use it to interpret the terms of the statement as expectation.

\smallskip\noindent$\bullet$  Let us fix $x_1$ and $x_2$. The rank $r^N(1,2)$ is equal to the number of points $x_3, \ldots, x_N$ belonging to the ball $B = B(x_1,|x_2-x_1|) = \{ x \, : \, |x-x_1| \leq |x_2-x_1| \}$ plus one unit, scaled by the factor $N-1$, {\it i.e.} 
$$  r^N(1,2) (\xx) = \frac{ \# \{ j \in \{3, \ldots, N\} \, : \, x_j \in B \} +1 }{N-1}. $$
Denote $P_R$ be the probability such that $R^N(1,2) = R$ where $R^N(1,2) = (N-1)  \,r^N(1,2) $. To have $R^N (1,2)= R$, we have to choose $R-1$ particles amongst $N-2$ to lie in $B$. 
The probability that one of the $R-1$ particles belongs to $B$ is equal to 
$$p:= M_{\rho^{(1)}_N} (x_1, |x_2-x_1|)\,,$$
while the probability that one of the $N-2 - (R-1)$ remaining particles lies in ${\mathbb R}^n \setminus B$ is $1-p$. Therefore, 
\begin{equation}
  \label{eq:pr}
  P_R = {N-2 \choose R-1} \, p^{R-1} \, (1-p)^{N-2-(R-1)}\;.
\end{equation}

\smallskip\noindent$\bullet$ Now, $x_1$ and $x_2$ being fixed, the quantity
$$\int K\left(r^N(1,2)(\xx)\right) \, \prod_{\ell = 3}^N \rho^{(1)}_N (x_\ell) \dd x_\ell, $$
can be interpreted as the expectation of $K(r^N(1,2))$ when $N-2$ points $x_3, \ldots, x_N$ are drawn according to independent identically distributed probabilities with density $\rho^{(1)}_N (x) \dd x$. It will be denoted ${\mathbb E} \{ K(r^N(1,2)(\xx)) \}
$.

By~\eqref{eq:pr}, we compute
\begin{align*}
{\mathbb E} \left\{ K\left(r^N(1,2)(\xx)\right) \right\}&= \sum_{R=1}^{N-1} K\left(\frac{R}{N-1} \right) {N-2 \choose R-1} \, p^{R-1} \, (1-p)^{N-2-(R-1)}\\
&= \sum_{R=0}^{M} K\left(\frac{R+1}{M+1} \right) {M \choose R} \, p^R \, (1-p)^{M-R}\;,
\end{align*}
with $M = N-2$. Since, for $N$ large, $K\left({(R+1)}/{(M+1)} \right) = K\left({R}/{M} \right)+o(1)$ (remarking that $R/M \leq 1$),
\begin{equation*}
{\mathbb E} \left\{ K\left(r^N(1,2)(\xx)\right) \right\}= \sum_{R=0}^{M} K\left(\frac{R}{M} \right) {M \choose R} \, p^R \, (1-p)^{M-R}+o(1)\;.
\end{equation*}
Using Bernstein's approximation, Proposition~\ref{bernstein}, we obtain
\begin{equation*}
{\mathbb E} \left\{ K\left(r^N(1,2)(\xx)\right) \right\}= K(p)+o(1)\;.
\end{equation*}
Which is the first statement.

\smallskip\noindent$\bullet$ The identity
\begin{equation*}
 \int  K\left(r^N(2,1)(\xx)\right)  \prod_{\ell = 3}^N \rho^{(1)}_N (x_\ell) \dd x_\ell =K\left(M_{\rho^{(1)}_N} (x_1, |x_2-x_1|)\right)+o(1), 
\end{equation*}
is obtained in an analogous way by exchanging the role of~$1$ and~$2$. We then have to integrate by $\rho^{(1)}_N (x_2) \dd x_2$ to obtain the stated result.
\end{proof}
\begin{proof}[of Proposition~\ref{cor:anbn}]
Inserting the expressions of Lemma~\ref{lem:1221} in $(A^N)$ and $(B^N)$ we readily obtain the stated result.
\end{proof}
\subsection{Evaluation of $(C^N) + (D^N)$}
\begin{proposition}[Evaluation of $(C^N) + (D^N)$]\label{lem:cd} Under the assumptions of Theorem~\ref{thm:limit}, we have
  \begin{multline*}
(C^N)+(D^N)=-\int \phi(Z_1) \, f^{(1)}_N (Z_1)\dd Z_1 \\-\int \phi(Z_1) \, f^{(1)}_N (Z_1)\,\rho^{(1)}_N(x_2)\,K\left( M_{\rho^{(1)}_N}(x_2,|x_1-x_2|\right)\dd x_2 \dd Z_1+o(1).
\end{multline*}
\end{proposition}

Like in the previous section, to evaluate $(C^N)$ we first transform the term in parenthesis:
\begin{lemma}\label{lem:1a} Under the assumptions of Theorem~\ref{thm:limit}, we have for large $N$
  \begin{multline*}
\int K\left(r^N(2,3)(\xx)\right)  \prod_{\ell=4}^N \rho^{(1)}_N(x_\ell) \dd x_\ell \\
\hspace{-1.5cm} = K(p_{23})- \frac{K'(p_{23})}{N} \big(1-\chi_{B(x_2,|x_2-x_3|)}(x_1)\big) \\+\frac{1}{N} \left[ \frac{p_{23}(1-p_{23})}{2}K''(p_{23})+2(1-p_{23})K'(p_{23})\right]+ o\left(\frac{1}{N}\right), 
\end{multline*}
where $p_{23}=M_{\rho^{(1)}_N} (x_2, |x_2-x_3|)$ only depends on $x_2$ and $x_3$ and $$\chi_{B(x_2,|x_2-x_3|)}(x_1)=
\left\{
  \begin{array}{ll}
    1 \quad&\mbox{if $x_1 \in B(x_2,|x_2-x_3|)$}\\
0\quad&\mbox{otherwise.}
  \end{array}
\right.
$$
\end{lemma}
\begin{proof}
Similarly to the proof of Lemma~\ref{lem:1221}, we interpret the quantity
\begin{equation*}
  \int K\left(r^N(2,3)(\xx)\right)  \prod_{\ell=4}^N \rho^{(1)}_N(x_\ell) \dd x_\ell , 
\end{equation*}
as the expectation of $K\left(r^N(2,3)(\xx)\right)$ when the $N-4$ points $x_4, \ldots, x_N$ are drawn according to independent identically distributed probabilities with density $\rho^{(1)}_N (x) \dd x$. Two cases have to be distinguished:\smallskip\\
\smallskip\noindent$\bullet$ {\bf First case: if $x_1 \in B(x_2,|x_2-x_3|)$ --} Like in the proof of Lemma~\ref{lem:1221} we have
\begin{equation*}
  r^N(2,3)=\frac{ \# \{ j \in \{4 \ldots, N\} \, |:\, x_j \in B(x_2,|x_2-x_3|) \} +2 }{N-1}.
\end{equation*}
Hence, setting $p_{23}=M_{\rho^{(1)}_N} (x_2, |x_2-x_3|)=:p$,
\begin{align*}
 \EE \left\{ K\left(r^N(2,3)(\xx)\right)\right\} &= \sum_{R=2}^{N-1} K\left(\frac{R}{N-1} \right) {N-3 \choose R-2} \, p^{R-2} \, (1-p)^{N-3-(R-2)}\\
&=\sum_{R=0}^{M} K\left(\frac{R+2}{M+2} \right) {M \choose R} \, p^{R} \, (1-p)^{M-R}, 
\end{align*}
with $M=N-3$. By expanding $K$, we have, uniformly with respect to $R \in \{0,\cdots,M\}$
\begin{equation*}
  K\left(\frac{R+2}{M+2} \right)=K\left(\frac{R}{M} \right)+\frac{2}{M}\left(\frac{M-R}{M+2} \right)K'\left(\frac{R}{M} \right) + o\left(\frac{1}{M} \right)\;.
\end{equation*}
Since $(M-R)/(M+2)=1-M/R + o(1)$ we obtain uniformly with respect to $R \in \{0,\cdots,M\}$
\begin{equation*}
  K\left(\frac{R+2}{M+2} \right)=K\left(\frac{R}{M} \right)+\frac{2}{M}\left(1-\frac{R}{M} \right)K'\left(\frac{R}{M} \right) + o\left(\frac{1}{M} \right)\;.
\end{equation*}
So, we obtain
\begin{multline*}
 \EE \left\{ K\left(r^N(2,3)(\xx)\right)\right\} =\sum_{R=0}^{M} K\left(\frac{R}{M} \right) {M \choose R} \, p^{R} \, (1-p)^{M-R} \\
+\frac{2}{M}\sum_{R=0}^{M}\left(1-\frac{R}{M} \right) K'\left(\frac{R}{M} \right) {M \choose R} \, p^{R} \, (1-p)^{M-R}+ o\left(1 \right).
\end{multline*}
Using Bernstein's approximation, Proposition~\ref{bernstein}, to $K$ and to $p \mapsto (1-p)K'(p)$ we obtain
\begin{align}
 \EE \left\{ K\left(r^N(2,3)(\xx)\right)\right\} &=K(p)+ \frac{p(1-p)}{2M}K''(p)  +\frac{2(1-p)}{M}K'(p)+ o\left(1 \right)\notag\\
&=K(p)+\frac{2(1-p)}{N}K'(p) + \frac{p(1-p)}{2N}K''(p)+ o\left(1\right). \label{eq:samep1}
\end{align}

\smallskip\noindent$\bullet$ {\bf Second case: if $x_1 \notin B(x_2,|x_2-x_3|)$ --} In this case, 
\begin{equation*}
  r^N(2,3)=\frac{ \# \{ j \in \{4 \ldots, N\} \, : \, x_j \in B(x_2,|x_2-x_3|) \} +1 }{N-1}.
\end{equation*}
Following the same step as before we compute, with $p=M_{\rho^{(1)}_N} (x_2, |x_2-x_3|)$
\begin{equation*}
 \EE \left\{K\left(r^N(2,3)(\xx)\right)\right\} = \sum_{R=1}^{N-2} K\left(\frac{R}{N-1} \right) {N-3 \choose R-1} \, p^{R-1} \, (1-p)^{N-3-(R-1)}\;,
\end{equation*}
which we rewrite
\begin{equation*}
 \EE \left\{\left(r^N(2,3)(\xx)\right)\right\} =\sum_{R=0}^{M} K\left(\frac{R+1}{M+2} \right) {M \choose R} \, p^{R} \, (1-p)^{M-R}, 
\end{equation*}
with $M=N-3$. By expanding $K$, we have
\begin{equation*}
  K\left(\frac{R+1}{M+2} \right)=K\left(\frac{R}{M} \right)+\frac{1}{M}\left(1-\frac{2R}{M} \right)K'\left(\frac{R}{M} \right) + o\left(\frac{1}{M} \right)\;.
\end{equation*}
So, we have 
\begin{multline*}
 \EE \left\{\left(r^N(2,3)(\xx)\right)\right\} =\sum_{R=0}^{M} K\left(\frac{R}{M} \right) {M \choose R} \, p^{R} \, (1-p)^{M-R} \\
+\frac{1}{M}\sum_{R=0}^{M}\left(1-\frac{2R}{M} \right) K'\left(\frac{R}{M} \right) {M \choose R} \, p^{R} \, (1-p)^{M-R}+ o\left(1\right).
\end{multline*}
Using Bernstein's approximation (see Proposition~\ref{bernstein}), we obtain
\begin{align}
 \EE \left\{\left(r^N(2,3)(\xx)\right)\right\} &=K(p)+ \frac{p(1-p)}{2M}K''(p)+\frac{1-2p}{M}K'(p)+ o\left(1 \right)\notag\\
&\hspace{-2cm}=K(p)+\frac{2(1-p)}{N}K'(p)- \frac{K'(p)}{N}+\frac{p(1-p)}{2N}K''(p)+o\left(1\right). \label{eq:samep}
\end{align}

\smallskip\noindent$\bullet$ We obtain the result stated in Lemma \ref{lem:1a} by noticing that Expression~\eqref{eq:samep} is equal to the sum of Expression~\eqref{eq:samep1} and an extra term $-K'(p)/N$.
\end{proof}
\begin{lemma}[Evaluation of $S^N(K)\times(C^N)/(N-2)$]\label{cor:sncn}
  Under the assumptions of Theorem~\ref{thm:limit}, we have
\begin{multline*}%
  \frac{S^N(K)}{N-2}\times (C^N)= \left(1+\frac1N+\frac{K(1)-K(0)}{2N}  \right)\int \phi(Z_1) \, f^{(1)}_N (Z_1)\dd Z_1\\
- \frac1N\int \phi(Z_1) \, f^{(1)}_N (Z_1)\,\rho_N^{(1)}(x_2)\,K\left( M_{\rho_N^{(1)}}(x_2,|x_2-x_3|\right)\dd x_2 \dd Z_1 +o(1).
\end{multline*}
\end{lemma}
\begin{proof}
  Using Lemma~\ref{lem:1a} and separating the cases $x_1 \in B(x_2,|x_2-x_3|)$ and $x_1 \not \in B(x_2,|x_2-x_3|)$, we can write
 \begin{equation*}
 \frac{S^N(K)}{N-2}\times(C^N)
= (1)+(2)+ o\left(1\right),
 \end{equation*}
where, writing $p$ for $p_{23}$, i.e. $p=M_{\rho^{(1)}_N} (x_2, |x_2-x_3|)$, we have:
\begin{multline*}
(1)=\frac1N\int_{x_1 \in B(x_2,|x_2-x_3|)}\!\!\!\!\!\!\!\!\! \phi(Z_1) \, f^{(1)}_N (Z_1) K'(p)\,\rho_N^{(1)}(x_2)\,\rho_N^{(1)}(x_3)\dd x_2 \dd x_3 \dd Z_1 \\
 + o\left( \frac{1}{N} \right), 
\end{multline*}
and
\begin{multline*}
  (2)=\int \phi(Z_1) \, f^{(1)}_N (Z_1) \left(K(p)+\frac{1-2p}{N}K'(p)+\frac{p(1-p)}{2N}K''(p)\right)\\
\rho_N^{(1)}(x_2)\,\rho_N^{(1)}(x_3)\dd x_2 \dd x_3 \dd Z_1 + o\left( \frac{1}{N} \right). 
\end{multline*}

\smallskip\noindent$\bullet$ For the term $(1)$, we first notice that $x_1 \in B(x_2,|x_2-x_3|)$ is equivalent to saying that $x_3 \notin B(x_2,|x_1-x_2|)$ so that, for $p=M_{\rho^{(1)}_N} (x_2, |x_2-x_3|)$, 
\begin{eqnarray*}
(1) &=& \frac1N\int_{x_3 \notin B(x_2,|x_1-x_2|)} \phi(Z_1) \, f^{(1)}_N (Z_1) K'(p)\,\rho_N^{(1)}(x_2)\,\rho_N^{(1)}(x_3)\dd x_2 \dd x_3 \dd Z_1
\\	
&& \hspace{9cm} + o\left( \frac{1}{N} \right) 
\;.
\end{eqnarray*}
By the change of variable stated in Lemma~\ref{lem:change} and applied to $H=K'$, $\rho=\rho_N^{(1)}$, $x=x_2$, and $r= |x_1-x_2|$ we have
\begin{multline*}
\int_{x_3 \notin B(x_2,|x_1-x_2|)} K'\left(M_{\rho^{(1)}_N} (x_2, |x_2-x_3|)\right)\rho_N^{(1)}(x_3)\dd x_3 \\=K(1)-K\left( M_{\rho_N^{(1)}}(x_2,|x_1-x_2|\right)\;.
\end{multline*}
Inserting this in $(1)$ we obtain
\begin{multline}\label{eq:a1}
  N\times (1)=
K(1)\int \phi(Z_1) \, f^{(1)}_N (Z_1)\dd Z_1 \int \rho_N^{(1)}(x_2) \dd x_2\\-  \int \phi(Z_1) \, f^{(1)}_N (Z_1)\,\rho_N^{(1)}(x_2)\,K\left( M_{\rho_N^{(1)}}(x_2,|x_1-x_2|\right)\dd x_2 \dd Z_1 + o(1)\\
\hspace{-4cm} = K(1)\int \phi(Z_1) \, f^{(1)}_N (Z_1)\dd Z_1\\-  \int \phi(Z_1) \, f^{(1)}_N (Z_1)\,\rho_N^{(1)}(x_2)\,K\left( M_{\rho_N^{(1)}}(x_2,|x_1-x_2|\right)\dd x_2 \dd Z_1 + o(1). 
\end{multline}

\smallskip\noindent$\bullet$ Using again the change of variable result of Lemma~\ref{lem:change} together with integration by parts, we compute:
\begin{multline*}
  \int \Big(K(p)+ \frac{1-2p}{N}K'(p)+\frac{p(1-p)}{2N}K''(p)\Big) \rho_N^{(1)}(x_3)\dd x_3  \\
=\int_0^1 \big( K(\tilde p)+\frac{1-2\tilde p}{N}K'(\tilde p)+\frac{\tilde p(1-\tilde p)}{2N}K''(\tilde p)\big) \, d\tilde p\\
= 1+\frac1N-\frac{K(0)+K(1)}{2N}\;,
\end{multline*}
where $p=M_{\rho^{(1)}_N} (x_2, |x_2-x_3|)$.
From this and Lemma \ref{lem:1a}, we deduce:
\begin{equation}
\label{eq:a2}
(2) = \left(1+\frac1N-\frac{K(0)+K(1)}{2N}\right)\int \phi(Z_1) \, f^{(1)}_N (Z_1) \dd Z_1 + o \left( \frac{1}{N} \right) . 
\end{equation}

\smallskip\noindent$\bullet$ Combining the two terms~\eqref{eq:a1} and~\eqref{eq:a2}, we obtain the result stated in Lemma \ref{cor:sncn}. 
\end{proof}
We are now ready to prove Proposition~\ref{lem:cd}
\begin{proof}[of Proposition~\ref{lem:cd}] The proof is divided in two main steps.

\noindent$\bullet$ We first have
\begin{align*}
  S^N(K)&=\frac{1}{N-1} \sum_{k=1}^{N-1}K \left( \frac{k}{N-1}\right)\\
&= \frac{K(1)-K(0)}{2(N-1)} + \frac{1}{N-1} \left(\frac{K(0)+K(1)}{2}+\sum_{k=1}^{N-2}K \left( \frac{k}{N-1}\right) \right)\;.
\end{align*}
In the second term of this expression, we recognize the approximation of $\int_0^1K(s)\dd s$ by the trapezoidal rule. As the trapezoidal rule is second order, it leads to
\begin{align*}
  S^N(K) &= \frac{K(1)-K(0)}{2(N-1)}  + \int_0^1K(s)\dd s +o\left( \frac1N\right) \\
&= \frac{K(1)-K(0)}{2N}  + 1 +o\left( \frac1N\right)\;.
\end{align*}
As a consequence
 \begin{align}\label{eq:par}
   \frac{N-2}{S^N(K)}&=N \frac{1-2/N}{1+(K(1)-K(0))/2N + o(1/N)}\notag\\
&=N-\frac{K(1)-K(0)}{2} -2+o(1)\;.
 \end{align}

\noindent $\bullet$ Now collecting the estimate of Corollary~\ref{cor:sncn} and~\eqref{eq:par} we obtain
\begin{align*}
  (C^N)=& \frac{N-2}{S^N(K)} \left[(1)+(2)\right]\\
=&\left(N- \frac{K(1)-K(0)}{2} -2\right)\int \phi(Z_1) \, f^{(1)}_N (Z_1)\dd Z_1 \\
&+ \left(1+\frac{K(1)-K(0)}{2} \right)\int \phi(Z_1) \, f^{(1)}_N (Z_1)\dd Z_1 \\
& - \int \phi(Z_1) \, f^{(1)}_N (Z_1)\,\rho_N^{(1)}(x_2)\,K\left( M_{\rho_N^{(1)}}(x_2,|x_1-x_2|\right)\dd x_2 \dd Z_1 \\
&+ o\left(1\right) \;.
\end{align*}
And as
\begin{equation*}
  (D^N)=-N\int \phi(Z_1) \, f^{(1)}_N (Z_1)\dd Z_1 ,
\end{equation*}
we obtain the statement of Proposition~\ref{lem:cd}. 
\end{proof}
\subsection{Proof of Theorem~\ref{thm:limit}}
We have to pass to the limit in~\eqref{eq:master2_poc}. By Propositions~\ref{cor:anbn} and~\ref{lem:cd} we have
\begin{align*}
  &\partial_t \int f^{(1)}_N(Z_1)\,\phi(Z_1) \dd Z_1 - \sum_{i=1}^N \int f^{(1)}_N(\zz)\, (v_i \cdot \nabla_{x_i})\,\phi(\zz)\dd \zz \\
&=\left( \frac{1}{S^N(K)} -1 \right)\!\int\! \phi(Z_1) \, f^{(1)}_N (Z_1)  \,  K\left(M_{\rho^{(1)}_N} (x_2, |x_1\!-\!x_2|)\right)\rho^{(1)}_N (x_2) \dd Z_1 \dd x_2\\
&\quad+\frac{1}{S^N(K)}\int \phi(x_1,v_2) \, \rho^{(1)}_N (Z_1) \, f^{(1)}_N (Z_2) \,  K\left(M_{\rho^{(1)}_N} (x_1, |x_1\!-\!x_2|)\right)  \dd x_1 \dd Z_2 \\
&\quad -\int \phi(Z_1) \, f^{(1)}_N (Z_1)\dd Z_1 +o(1) .
\end{align*} 
As $N$ goes to $\infty$, the second line goes to $0$ since $S^N(K) $ is the Riemann sum approximation of $\int_0^1 K(r) \dd r= 1$. The convergence in the other terms is formally obvious and leads to the stated result.

\section{Discussion}\label{sec:conclusion}
\subsection{Large-time behaviour}
Consider a function homogeneous in space $(t,x,v) \mapsto G(t,v)$. Since $\int K=1$, by Lemma~\ref{lem:change}, we get
\begin{equation*}
  \frac{\partial G}{\partial t}(v)=-v\cdot \nabla_x G(v)+G(v)\int K\left( M_\rho (x,|x'-x|)\right)\dd x' - G(v)=0\;.
\end{equation*}
Hence any function homogeneous in space $(t,x,v) \mapsto G(t,v)$ is a stationary solution. 
Moreover, on a periodic spatial domain, we can expect that any solution converges at large-times toward a function of this type. The proof of such a claim is left to future work. 
\subsection{Discrete versus continuous approach}
We can wonder if the large-time and large number of particles limits permute. It does not seem the case. Indeed, the number of distinct velocities decreases when there is a finite number of particles while, as discussed in the previous section, the distribution of velocities remains constant in time in the case of a continuum of particles.

In the case of a finite number of particles the consensus in the direction the particles adopt is longer and longer to obtain, see Figures~\ref{fig:10},~\ref{fig:20} and~\ref{fig:70}.

%

\begin{figure}[h!]
 \begin{minipage}[t]{.45\linewidth}
\centering\epsfig{figure=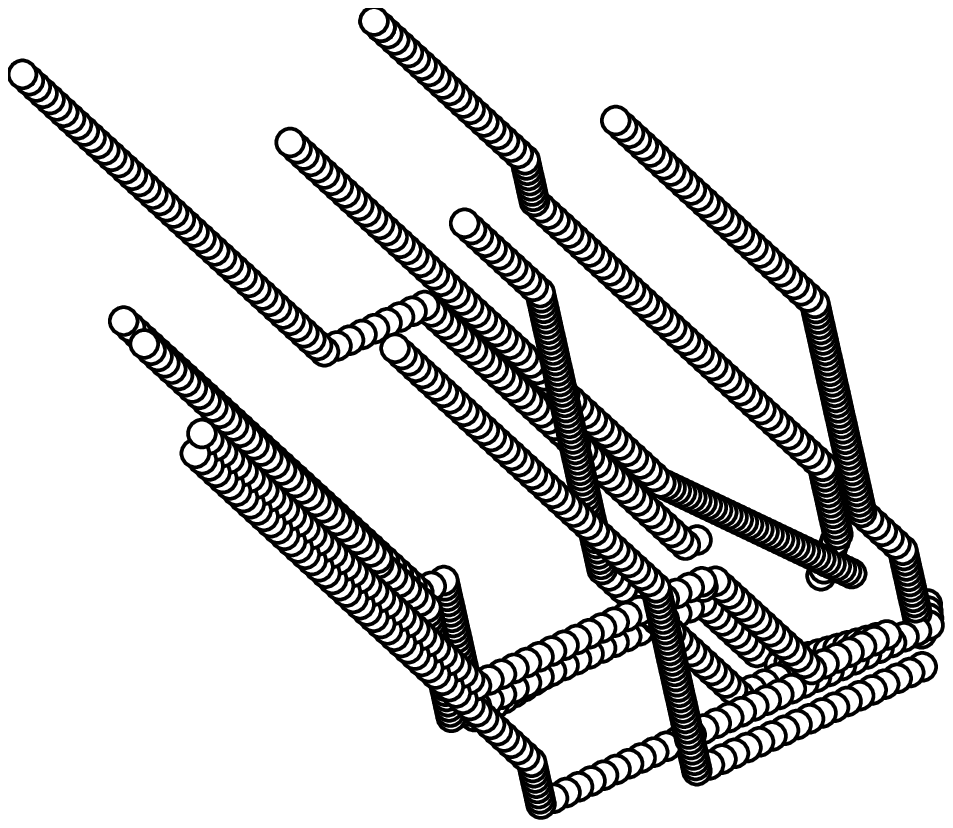,width=6cm}
 \end{minipage} \hfill
\begin{minipage}[t]{.45\linewidth}
\centering\epsfig{figure=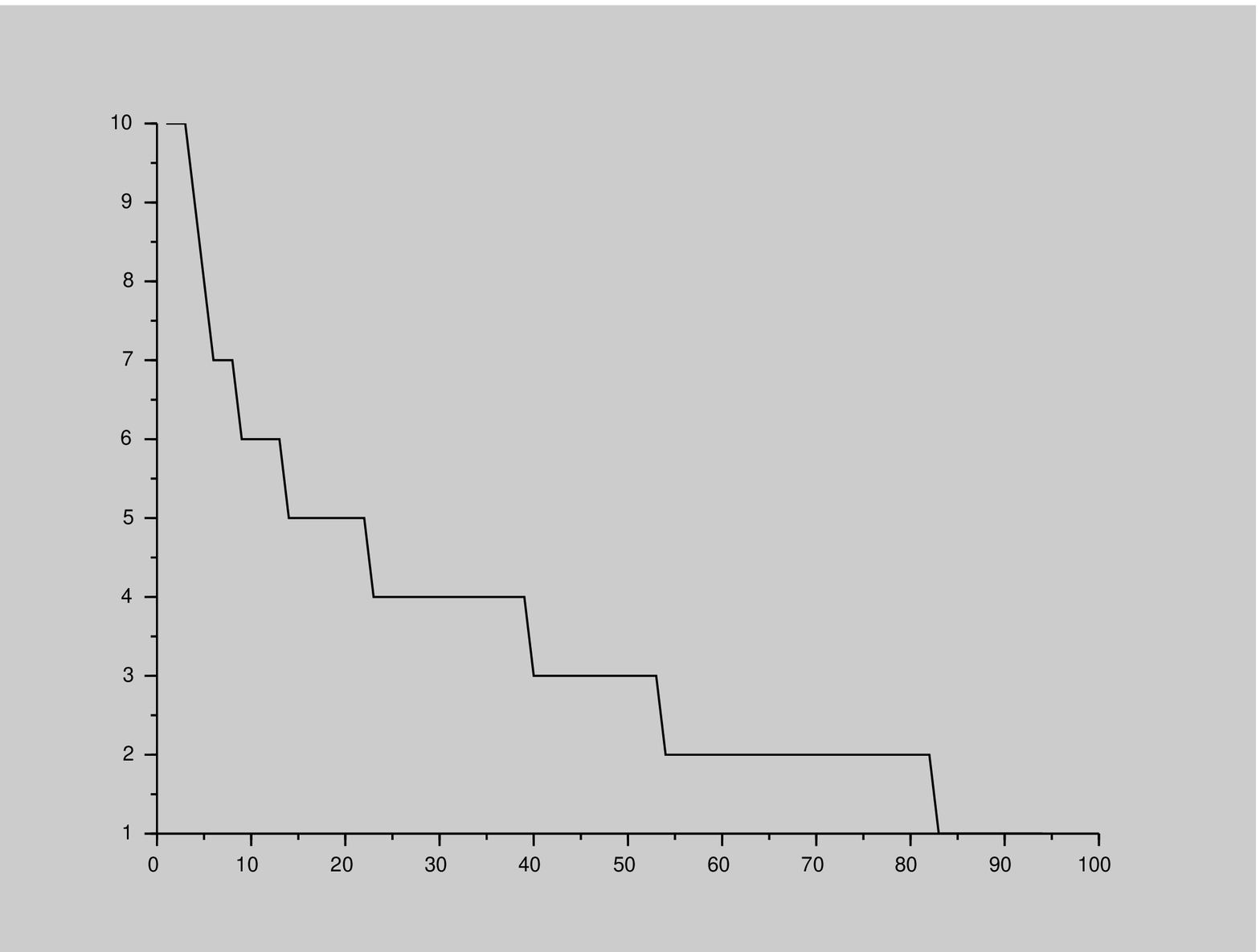,width=6cm}
 \end{minipage} 
\caption{\small Trajectories of the particles on the left, variance and number of different speed as functions of time in the case of 10 particles taken randomly in $[-10,10]$ for the position and for the speed.}\label{fig:10}
\end{figure}
\begin{figure}[h!]
 \begin{minipage}[t]{.45\linewidth}
\centering\epsfig{figure=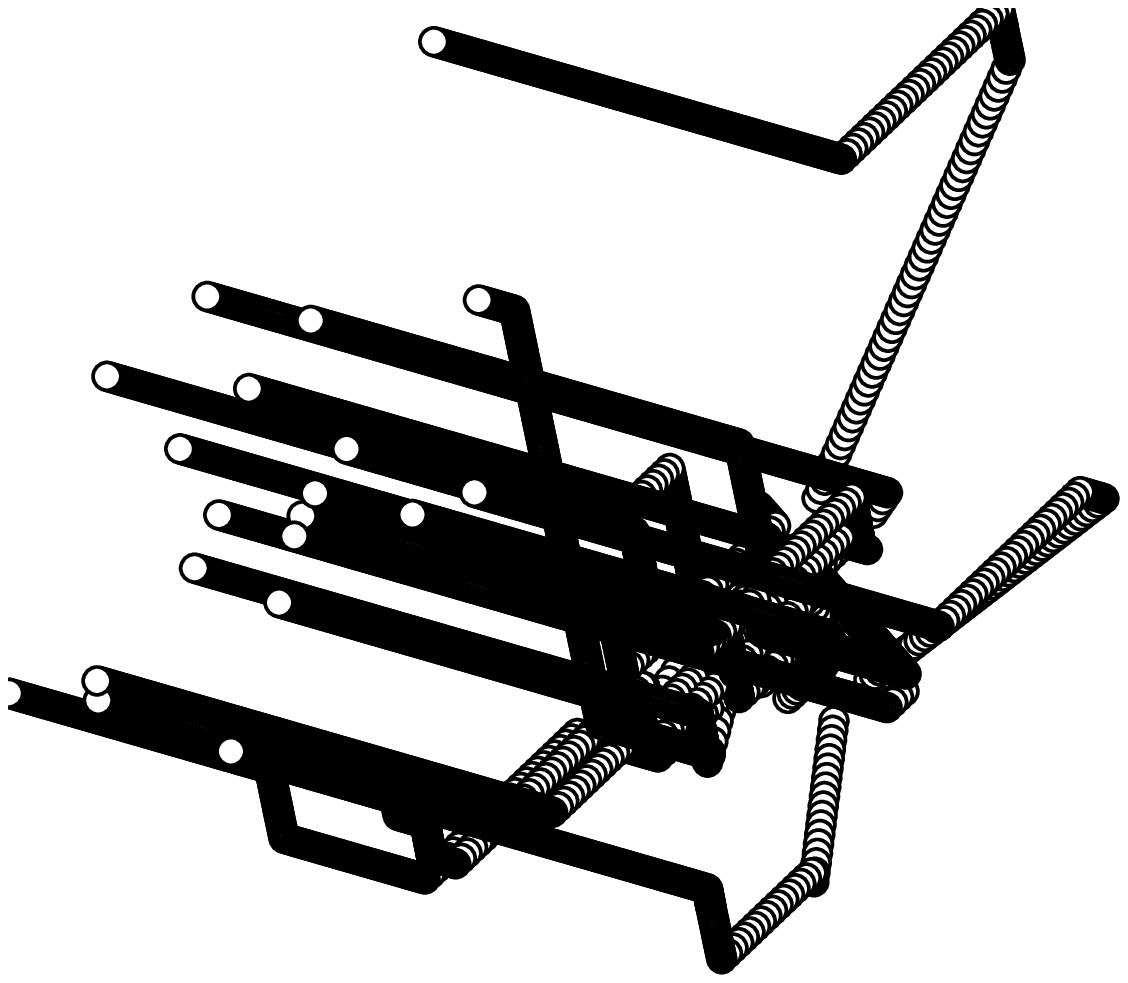,width=6cm}
 \end{minipage} \hfill
\begin{minipage}[t]{.45\linewidth}
\centering\epsfig{figure=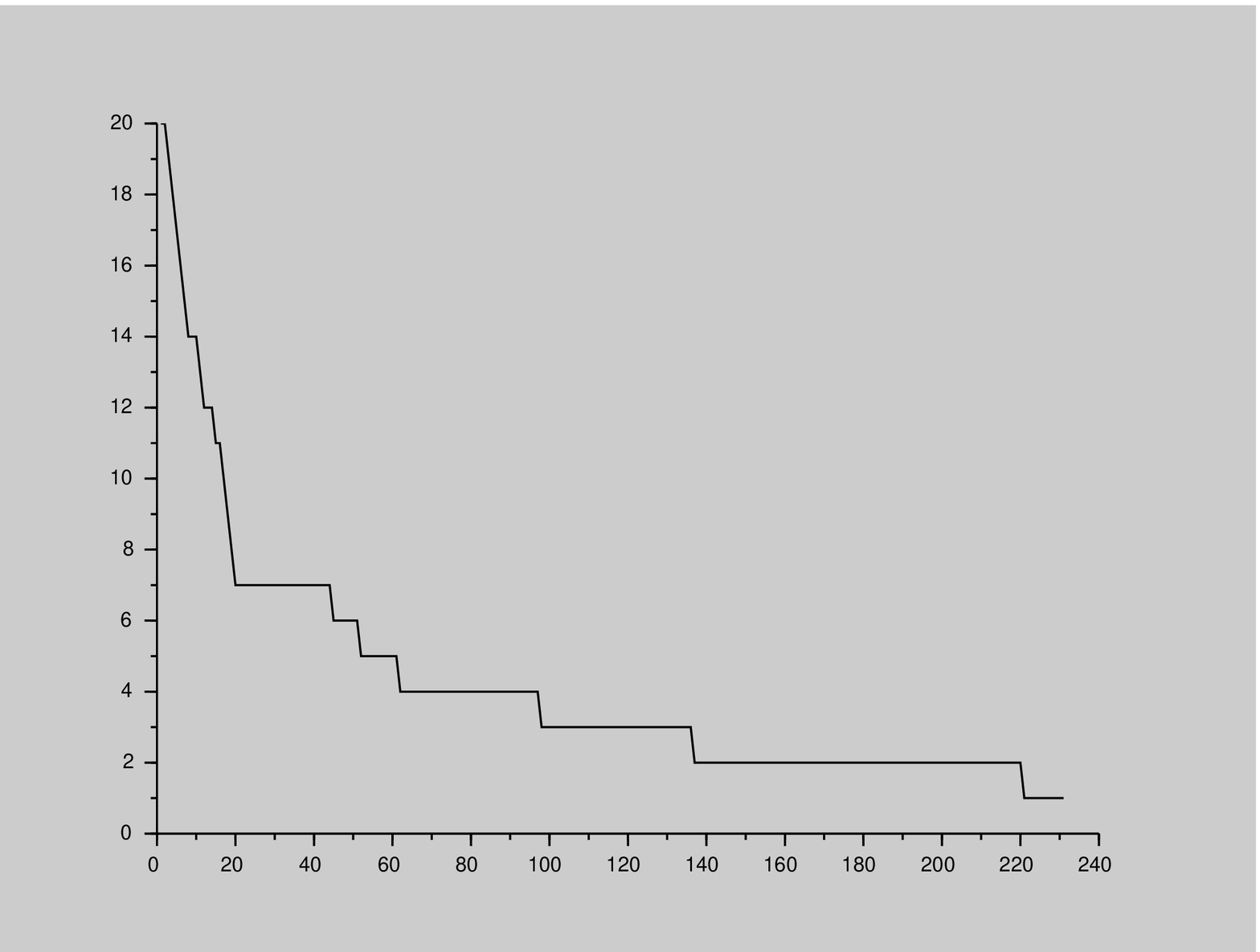,width=6cm}
 \end{minipage} 
\caption{\small Trajectories of the particles on the left, variance and number of different speed as functions of time in the case of 20 particles taken randomly in $[-10,10]$ for the position and for the speed.}\label{fig:20}
\end{figure}
\begin{figure}[h!]
 \begin{minipage}[t]{.45\linewidth}
\centering\epsfig{figure=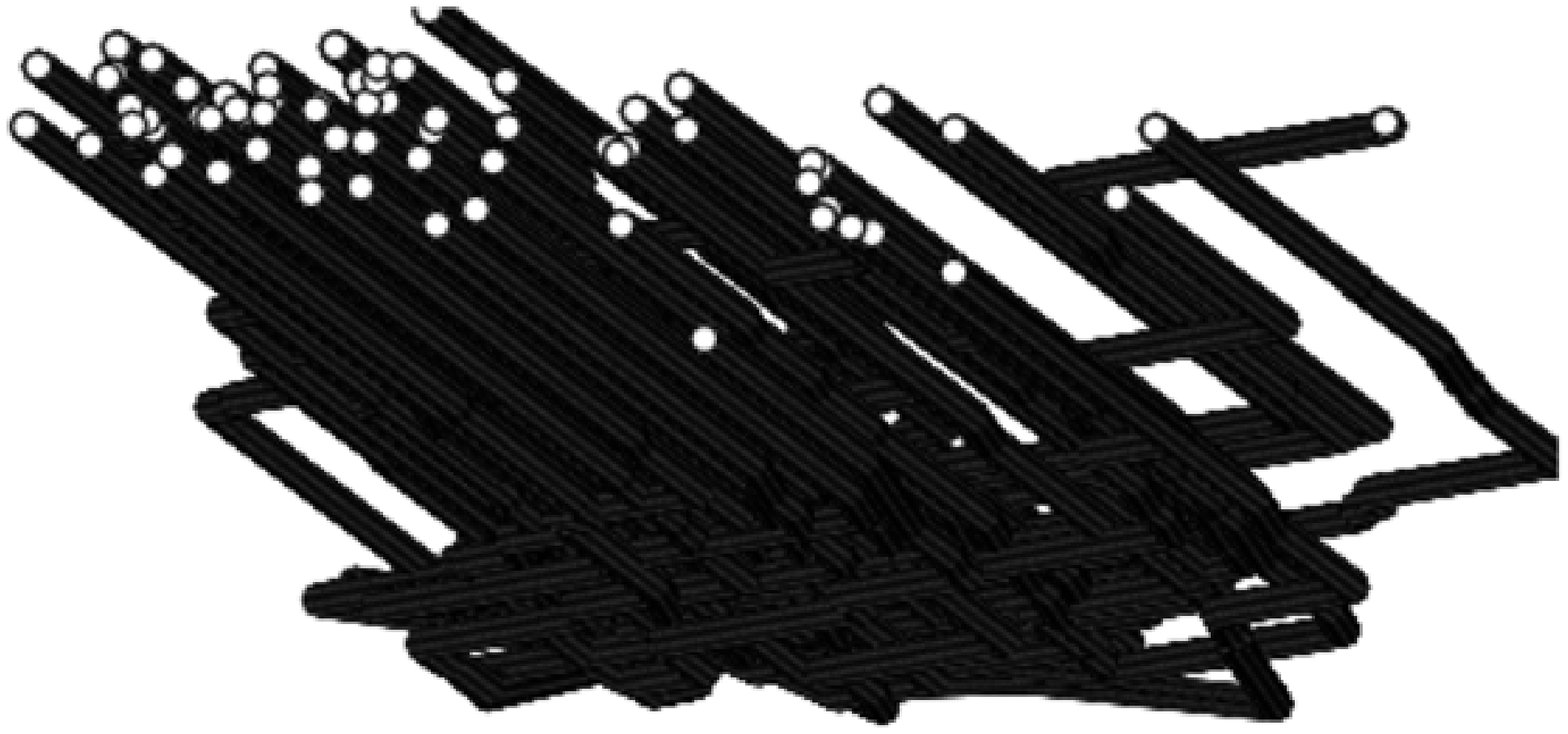,width=6cm}
 \end{minipage} \hfill
\begin{minipage}[t]{.45\linewidth}
\centering\epsfig{figure=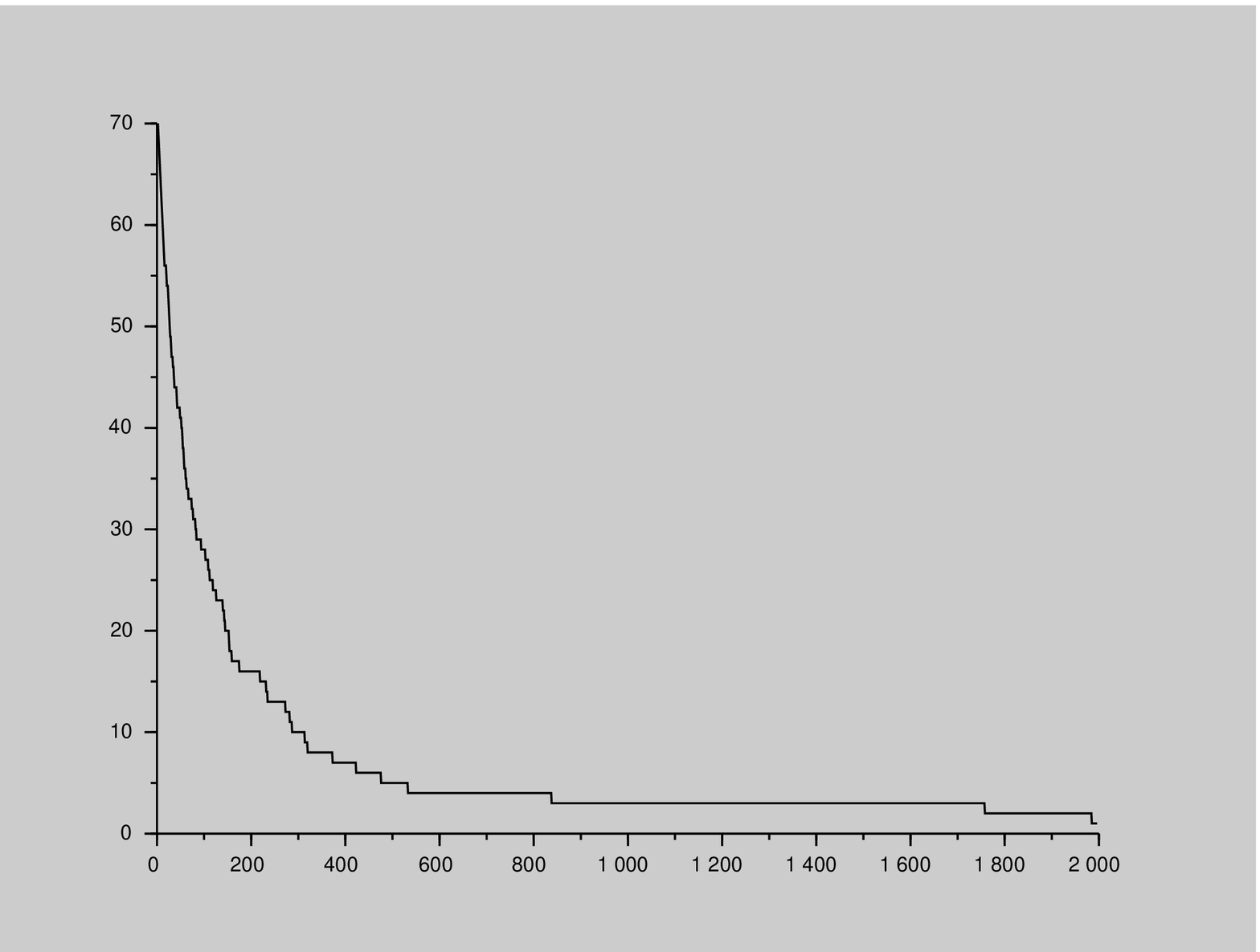,width=6cm}
 \end{minipage} 
\caption{\small Trajectories of the particles on the left, variance and number of different speed as functions of time in the case of 70 particles taken randomly in $[-10,10]$ for the position and for the speed.}\label{fig:70}
\end{figure}

\section{Conclusion}\label{sec:conclusion2}

In this paper, we have investigated a system of particles interacting through leader following interactions where the choice of the leader is determined by a topological rule. Under a propagation of chaos assumption, we have shown that the large system size limit is described by a spatially nonlocal kinetic model of Boltzmann type. This result heavily relies on approximation properties of Bernstein polynomials. Obviously, the very simple leader following model considered in this paper offers many directions of complexification leading to biologically or socially more realistic rules. An example could be the introduction  of some noise, e.g. the velocity after the interaction would be randomly selected according to a probability law centred around the leader velocity. One could also think of the two particles joining their average velocity up to some noise, in the spirit of \cite{BertinDrozGregoire2006}. Finally, binary interactions with the closest neighbour could also be investigated.

\appendix
\section{Fundamental lemma}
\begin{lemma}\label{lem:change}
  For any $H$,
\begin{equation*}
 I(x;r):=\int_{B_r(x)} H\!\left( M_\rho(x,|x'-x|)\right)\rho(x') \dd x' = \int_0^{M_\rho(x,r)} H(p)\dd p\;.
\end{equation*}
\end{lemma}
\begin{proof}
First note that since
\begin{equation*}
  M_\rho(x,s) = \int_{\tilde s<s}\int_{\omega \in \SS^{n-1}} \rho(x+\tilde s\,\omega) \,\tilde s^{n-1}\dd \tilde s\dd \omega, 
\end{equation*}
we have
\begin{equation*}
  \frac{\dd}{\dd s}M_\rho(x,s) = \int_{\omega \in \SS^{n-1}} \rho(x+s\,\omega) \,s^{n-1}\dd \omega.
\end{equation*}

Using the polar change of variables, 
\begin{equation*}
  |x'-x|=:s\qquad \frac{x'-x}{|x'-x|}=:\omega\;,
\end{equation*}
we have
\begin{align*}
 I(x;r) &= \int_{s<r}\int_{\omega \in \SS^{n-1}} \rho(x+s\,\omega) \,H\!\left( M_\rho(x,s)\right)s^{n-1} \dd s\dd \omega\\
  &=\int_{s<r}\frac{\dd}{\dd s}M_\rho(x,s) \,H\!\left( M_\rho(x,s)\right) \dd s\;.
\end{align*}
Setting $p=M_\rho(x,s)$, so that $\dd p=\frac{\dd}{\dd s}M_\rho(x,s) \dd s$, we obtain the stated result.
\end{proof}

\medskip
\noindent{\bf Acknowledgements} This work has been supported by the Engineering and Physical 
Sciences Research Council (EPSRC) under grant ref: 
EP/M006883/1, by the Agence Nationale pour
la Recherche (ANR) under grant MOTIMO (ANR-11-MONU-009-01)  
and by the National Science Foundation (NSF) under grant 
RNMS11-07444 (KI-Net). PD is on leave from CNRS, 
Institut de Math\'ematiques de Toulouse, France. 
PD gratefully acknowledges support from the Royal 
Society and the Wolfson foundation through a Royal 
Society Wolfson Research Merit Award.

\end{document}